\documentclass[12pt]{amsart}

\usepackage{amsmath}
\usepackage{amssymb}
\usepackage{amsfonts}
\usepackage{amsthm}
\usepackage{enumerate}
\usepackage{hyperref}
\usepackage{color}

\theoremstyle{plain}
\newtheorem{thm}{Theorem}[section]

\newtheorem{lem}[thm]{Lemma}
\newtheorem{cor}[thm]{Corollary}

\newtheorem{conj}[thm]{Conjecture}

\theoremstyle{definition}

\newtheorem{exmp}[thm]{Example}

\newtheorem{rem}[thm]{Remark}

\newtheorem{dfns-rems}[thm]{Definitions and Remarks}
\newtheorem{notas-rems}[thm]{Notations and Remarks}
\newtheorem{exmps-rems}[thm]{Examples and Remarks}

\begin{document}


\title[Stanley depth of the integral closure of monomial ideals]{Stanley depth of the integral closure of monomial ideals}


\author[S. A. Seyed Fakhari]{S. A. Seyed Fakhari}

\address{S. A. Seyed Fakhari, Department of Mathematical Sciences,
Sharif University of Technology, P.O. Box 11155-9415, Tehran, Iran.}

\email{fakhari@ipm.ir}

\urladdr{http://math.ipm.ac.ir/fakhari/}


\begin{abstract}
Let $I$ be a monomial ideal in the polynomial ring $S=\mathbb{K}[x_1,\dots,x_n]$. We study the Stanley depth of the integral closure $\overline{I}$ of $I$. We prove that for every integer $k\geq 1$, the inequalities ${\rm sdepth} (S/\overline{I^k}) \leq {\rm sdepth} (S/\overline{I})$ and ${\rm sdepth} (\overline{I^k}) \leq {\rm sdepth} (\overline{I})$ hold. We also prove that for every monomial ideal $I\subset S$ there  exist integers $k_1,k_2\geq 1$, such that for every $s\geq 1$, the inequalities  ${\rm sdepth} (S/I^{sk_1}) \leq {\rm sdepth} (S/\overline{I})$ and ${\rm sdepth} (I^{sk_2}) \leq {\rm sdepth} (\overline{I})$ hold. In particular,
$\min_k \{{\rm sdepth} (S/I^k)\} \leq {\rm sdepth} (S/\overline{I})$
and
$\min_k \{{\rm sdepth} (I^k)\} \leq {\rm sdepth} (\overline{I})$.
We conjecture that for every integrally closed monomial ideal $I$, the inequalities ${\rm sdepth}(S/I)\geq n-\ell(I)$ and ${\rm sdepth} (I)\geq n-\ell(I)+1$ hold, where $\ell(I)$ is the analytic spread of $I$. Assuming the conjecture is true, it follows together with the Burch's inequality that Stanley's conjecture holds for $I^k$ and $S/I^k$ for $k\gg 0$, provided that $I$ is a normal ideal.
\end{abstract}






\maketitle


\section{Introduction and preliminaries} \label{sec1}

Let $\mathbb{K}$ be a field. Throughout this paper, the
polynomial ring $\mathbb{K}[x_1,\dots,x_n]$ in $n$ variables over the field $\mathbb{K}$ is denoted by $S$.

Let $M$
be a finitely generated $\mathbb{Z}^n$-graded $S$-module.
Let $u\in M$ be a homogeneous element and $Z\subseteq \{x_1,\dots,x_n\}$. The $\mathbb
{K}$-subspace $u\mathbb{K}[Z]$ generated by all elements $uv$
with $v\in \mathbb{K}[Z]$ is called a {\it Stanley space} of dimension
$|Z|$, if it is a free $\mathbb{K}[Z]$-module. Here, as usual, $|Z|$ denotes the number of elements of $Z$. A decomposition $\mathcal{D}$ of $M$ as a finite direct sum of Stanley
spaces is called a {\it Stanley decomposition} of $M$. The minimum
dimension of a Stanley space in $\mathcal{D}$ is called {\it Stanley depth}
of $\mathcal{D}$ and is denoted by ${\rm sdepth}(\mathcal {D})$. The
quantity $${\rm sdepth}(M):=\max\big\{{\rm sdepth}(\mathcal{D})\mid
\mathcal{D}\ {\rm is\ a\ Stanley\ decomposition\ of}\ M\big\}$$ is called
{\it Stanley depth} of $M$. Stanley \cite{s} conjectured
that $${\rm depth}(M) \leq {\rm sdepth}(M)$$ for all $\mathbb{Z}^n$-graded $S$-modules $M$. For an introduction to Stanley depth, we refer the reader to \cite{psty}.

Let $I\subset S$ be an arbitrary ideal. An element $f \in S$ is
{\it integral} over $I$, if there exists an equation
$$f^k + c_1f^{k-1}+ \ldots + c_{k-1}f + c_k = 0 {\rm \ \ \ \ with} \ c_i\in I^i.$$
The set of elements $\overline{I}$ in $S$ which are integral over $I$ is the {\it integral closure}
of $I$. It is known that the integral closure of a monomial ideal $I\subset S$ is a monomial ideal
generated by all monomials $u \in S$ for which there exists an integer $k$ such that
$u^k\in I^k$ (see \cite[Theorem 1.4.2]{hh'}).

\begin{rem} \label{clo}
Let $I$ be a monomial ideal and let $G(\overline{I})=\{m_1, \ldots, m_s\}$ be the set of minimal monomial generators of $\overline{I}$. For every $1\leq i \leq s$, there exists integer $k_i\geq 1$ such that $m_i^{k_i}\in I^{k_i}$. Let $k={\rm lcm}(k_1, \ldots,k_s)$ be the least common multiple of $k_1, \ldots k_s$. Now for every $1\leq i \leq s$, we have $m_i^k\in I^k$ and this implies that $u^k\in I^k$, for every monomial $u\in \overline{I}$. It follows that for every monomial $u\in S$, we have $u\in \overline{I}$ if and only if $u^k\in I^k$.
\end{rem}

\ \

The ideal $I$ is {\it integrally closed}, if $I = \overline{I}$, and $I$ is {\it normal} if all powers
of $I$ are integrally closed. By \cite[Theorem 3.3.18]{v'}, a monomial ideal $I$ is normal if and only if the Rees algebra $\mathcal{R}(I)$ is a normal ring.

Apel \cite{a} proved that  for every monomial ideal $I\subset S$, the inequality ${\rm sdepth}(S/I)\leq {\rm sdepth} (S/\sqrt{I})$ holds (see also \cite{i'}). It is clear that for every monomial ideal $I\subset S$, we have $I\subseteq \overline{I} \subseteq \sqrt{I}$ and therefore $\sqrt{\overline{I}}=\sqrt{I}$ and hence ${\rm sdepth}(S/\overline{I})\leq {\rm sdepth} (S/\sqrt{I})$. Now it is natural to ask about the relation of ${\rm sdepth}(S/I)$ and ${\rm sdepth}(S/\overline{I})$. There is no general inequality between ${\rm sdepth}(S/I)$ and ${\rm sdepth}(S/\overline{I})$, as the following examples show.

\begin{exmp} \label{ex1}
Let $I=(x_1^2, x_2^2, x_1x_2x_3)$ be a monomial ideal in the polynomial ring $S=\mathbb{K}[x_1, x_2,x_3]$. Then one can easily check that $\overline{I}=(x_1^2, x_2^2,x_1x_2)$. The maximal ideal $\mathfrak{m}=(x_1,x_2,x_3)$ of $S$ is an associated prime of $S/I$ and therefore \cite[Proposition 1.3]{hvz} (see also \cite{a}) implies that ${\rm sdepth}(S/I)=0$. Since $\mathfrak{m}$ is not an associated prime of $S/\overline{I}$, it follows from \cite[Proposition 2.13]{br} that ${\rm sdepth}(S/\overline{I})\geq 1$. Thus in this example ${\rm sdepth}(S/I) < {\rm sdepth}(S/\overline{I})$.
\end{exmp}

\begin{exmp} \label{ex2}
Let $I=(x_1^2x_2^2, x_1^2x_3^2, x_2^2x_3^2)$ be a monomial ideal in the polynomial ring $S=\mathbb{K}[x_1, x_2, x_3]$. Then the maximal ideal $\mathfrak{m}=(x_1,x_2,x_3)$ of $S$ is not an associated prime of $S/I$ and therefore \cite[Proposition 2.13]{br} implies that ${\rm sdepth}(S/I)\geq 1$. On the other hand by \cite[Theorem 2.4]{j}, $\mathfrak{m}$ is an associated prime of $S/\overline{I}$ and therefore using \cite[Proposition 1.3]{hvz} (see also \cite{a}), it follows that ${\rm sdepth}(S/\overline{I})=0$. Thus in this example ${\rm sdepth}(S/I) > {\rm sdepth}(S/\overline{I})$.
\end{exmp}

Examples \ref{ex1} and \ref{ex2} show that there is no general inequality between the Stanley depth of $S/I$ and the Stanley depth of $S/\overline{I}$. However, we prove that for every monomial ideal $I\subset S$ there  exist integers $k_1,k_2\geq 1$, such that for every $s\geq 1$, the inequalities  ${\rm sdepth} (S/I^{sk_1}) \leq {\rm sdepth} (S/\overline{I})$ and ${\rm sdepth} (I^{sk_2}) \leq {\rm sdepth} (\overline{I})$ hold (Corollary \ref{mainc}). In particular
$$\min_k \{{\rm sdepth} (S/I^k)\} \leq {\rm sdepth} (S/\overline{I})$$
and
$$\min_k \{{\rm sdepth} (I^k)\} \leq {\rm sdepth} (\overline{I}).$$

Ratliff \cite{r2} proves that for every ideal $I$ in a commutative Noetherian ring $S$, the asymptotic set of associated primes of integral closure of powers of $I$ is a subset of the asymptotic set of associated primes of powers of $I$. We use Corollary \ref{mainc} to give a new proof for Ratliff's theorem in the case of monomial ideals (Theorem \ref{ass}).

We also prove that for every monomial ideal $I\subset S$, the inequalities ${\rm sdepth} (S/\overline{I^k}) \leq {\rm sdepth} (S/\overline{I})$ and ${\rm sdepth} (\overline{I^k}) \leq {\rm sdepth} (\overline{I})$ hold for every integer $k\geq 1$ (Theorem \ref{first}). This implies that for every normal monomial ideal $I$, there exists $k$, such that ${\rm sdepth} (S/I^k)={\rm sdepth} (S/I^{sk})$, for every integer $s\geq 1$.

In Section \ref{sec2}, we present a conjecture, regarding the Stanley depth of integrally closed monomial ideals. In order to do this, we need to introduce some notation and well known results.

Let $I$ be a monomial ideal of $S$ with Rees algebra $\mathcal{R}(I)$
and let $\mathfrak{m}=(x_1,\ldots,x_n)$ be the graded maximal ideal of $S$. Then
the $\mathbb{K}$-algebra $\mathcal{R}(I)/\mathfrak{m}\mathcal{R}(I)$ is
called the {\it fibre ring} and its Krull dimension is called the {\it
analytic spread} of $I$, denote by $\ell(I)$. This invariant is a measure
for the growth of the number of generators of the powers of $I$. Indeed,
for $k\gg 0$, the Hilbert function $H(\mathcal{R}(I)/\mathfrak{m}\mathcal{R}(I),\mathbb{K},k)={\rm dim}_\mathbb
{K}(I^k/\mathfrak{m}I^k)$, which counts the number of generators of the
powers of $I$, is a polynomial function of degree $\ell(I)-1$. Let $I\subset S$, be a monomial ideal. An ideal $J\subseteq I$ is called a {\it reduction} of $I$, if $JI^t=I^{t+1}$, for some integer $t\geq 1$. It is known by \cite[Corollaries 8.2.5 and 8.3.9]{hs} that for every ideal $I$ and every reduction $J$ of $I$, the inequality ${\rm ht}(I)\leq \ell(I)\leq \mu(J)$ holds, where $\mu(J)$ denotes the number of minimal generators of
$J$.

Let $I \subset S$ be a monomial ideal. A classical
result by Burch \cite{b1} says that $$\min_k{\rm depth}(S/I^k) \leq
n-\ell(I).$$ By a theorem of Brodmann \cite{b}, the quantity ${\rm
depth}(S/I^k)$ is constant for large $k$. We call this constant value the
{\it limit depth} of $I$ and we denote it by $\lim_{k\rightarrow \infty}{\rm
depth}(S/I^k)$. Brodmann improves Burch's inequality by showing
that$$\lim_{k\rightarrow \infty}{\rm depth}(S/I^k) \leq n-\ell(I).$$
with equality if the Rees algebra $\mathcal{R}(I)$ is a normal ring.
In Section \ref{sec2}, we conjecture that for every integrally closed monomial ideal $I$, the inequalities ${\rm sdepth}(S/I)\geq n-\ell(I)$ and ${\rm sdepth} (I)\geq n-\ell(I)+1$ hold. Assuming the conjecture is true, it follows together with the Burch's inequality that Stanley's conjecture holds for $I^k$ and $S/I^k$ for $k\gg 0$, provided that $I$ is a normal ideal.


\section{Stanley depth and integral closure of monomial ideals} \label{sec2}

Let $I$ be a monomial ideal. As the first result of this paper, we compare the Stanley depth of the integral closure of $I$ and the Stanley depth of the integral closure of powers of $I$.

\begin{thm} \label{first}
Let $J\subseteq I$ be two monomial ideals in $S$. Then for every integer $k\geq 1$
$${\rm sdepth} (\overline{I^k}/\overline{J^k}) \leq {\rm sdepth} (\overline{I}/\overline{J}).$$
\end{thm}

\begin{proof}
Let $u\in S$ be a monomial. Then $u\in \overline{I}$ if and only if $u^s\in I^s$, for some $s\geq 1$ if and only if $u^{ks'}\in I^{ks'}$, for some $s'\geq 1$ if and only if $u^k\in \overline{I^k}$. By a similar argument $u\in \overline{J}$ if and only if $u^k\in \overline{J^k}$.

Now consider a Stanley decomposition $$\mathcal{D} : \overline{I^k}/\overline{J^k}=\bigoplus_{i=1}^m t_i \mathbb{K}[Z_i]$$ of $\overline{I^k}/\overline{J^k}$, such that ${\rm sdepth}(\mathcal{D})={\rm sdepth} (\overline{I^k}/\overline{J^k})$.  By the argument above, for every monomial $u\in \overline{I}\setminus \overline{J}$, we have $$u^k\in \overline{I^k}\setminus \overline{J^k}.$$Thus for each monomial $u\in \overline{I}\setminus \overline{J}$, we define $Z_u:=Z_i$ and $t_u:=t_i$, where $i\in\{1, \ldots, m\}$ is the uniquely determined index, such that $u^k\in t_i \mathbb{K}[Z_i]$. It is clear that $\overline{I}/\overline{J}\subseteq \sum u\mathbb{K}[Z_u]$, where the sum is taken over all monomials $u\in \overline{I}\setminus \overline{J}$. For the converse inclusion note that for every $u\in \overline{I}\setminus \overline{J}$ and every $h\in \mathbb{K}[Z_u]$, clearly we have $uh\in \overline{I}$. By the choice of $t_u$ and $Z_u$, we conclude $u^k\in t_u \mathbb{K}[Z_u]$ and therefore $u^kh^k\in t_u \mathbb{K}[Z_u]$. This implies that $u^kh^k\notin \overline{J^k}$ and as argument above shows, $uh\notin \overline{J}$. Thus $$\overline{I}/\overline{J}= \sum u\mathbb{K}[Z_u],$$ where the sum is taken over all monomials $u\in \overline{I}\setminus \overline{J}$.

Now for every $1\leq i \leq m$, let
$$U_i=\{u\in \overline{I}\setminus \overline{J}: Z_u=Z_i \ {\rm and} \ t_u=t_i\}.$$
Without lose of generality we may assume that $U_i\neq \emptyset$ for every $1\leq i \leq l$ and $U_i=\emptyset$ for every $l+1\leq i\leq m$. For every $1\leq i \leq l$, let $u_i$ be the greatest common divisor of elements of $U_i$. Note that $$\overline{I}/\overline{J}= \sum_{i=1}^l \sum u\mathbb{K}[Z_i],$$ where the second sum is taken over all monomials $u\in U_i$. Since for every $u\in U_i$, $u^k\in t_i \mathbb{K}[Z_i]$, it follows that $u_i^k\in t_i \mathbb{K}[Z_i]$. Therefore $u_i \in \overline{I}\setminus \overline{J}$ and hence $u_i\in U_i$. Now for every $u\in U_i$, we have $u\mathbb{K}[Z_i]\subseteq u_i\mathbb{K}[Z_i]$ and thus $$\sum_{u\in U_i} u\mathbb{K}[Z_i]=u_i\mathbb{K}[Z_i].$$It follows that $$\overline{I}/\overline{J}= \sum_{i=1}^l u_i\mathbb{K}[Z_i].$$

Next we prove that for every $1\leq i,j \leq l$ with $i\neq j$, the summands $u_i\mathbb{K}[Z_i]$ and $u_j\mathbb{K}[Z_j]$ intersect trivially. By contradiction let $v$ be a monomial in $u_i\mathbb{K}[Z_i]\cap u_j\mathbb{K}[Z_j]$. Then there exist $h_i\in \mathbb{K}[Z_i]$ and $h_j\in \mathbb{K}[Z_j]$, such that $u_ih_i=v=u_jh_j$. Therefore $u_i^kh_i^k=v^k=u_j^kh_j^k$. But $u_i\in U_i$ and hence $u_i^k\in t_i \mathbb{K}[Z_i]$, which implies that $u_i^kh_i^k\in t_i \mathbb{K}[Z_i]$. Similarly $u_j^kh_j^k\in t_j \mathbb{K}[Z_j]$. Thus $$v^k\in t_i \mathbb{K}[Z_i]\cap t_j \mathbb{K}[Z_j],$$ which is a contradiction, because $\bigoplus_{i=1}^m t_i \mathbb{K}[Z_i]$ is a Stanley decomposition of $\overline{I^k}/\overline{J^k}$. Therefore  $$\overline{I}/\overline{J}=\bigoplus_{i=1}^l u_i \mathbb{K}[Z_i]$$ is a Stanley decomposition of $\overline{I}/\overline{J}$ which proves ${\rm sdepth} (\overline{I}/\overline{J})\geq \min_{i=1}^l|Z_i|\geq{\rm sdepth} (\overline{I^k}/\overline{J^k})$.
\end{proof}

The following corollaries are immediate consequences of Theorem \ref{first}

\begin{cor}
Let $I \subset S$ be a monomial ideal. Then for every integer $k\geq 1$,
$${\rm sdepth} (\overline{I^k}) \leq {\rm sdepth} (\overline{I})$$
and
$${\rm sdepth} (S/\overline{I^k}) \leq {\rm sdepth} (S/\overline{I}).$$
\end{cor}

\begin{cor} \label{nor}
Let $I\subset S$ be a normal monomial ideal. Then for every integer $k\geq 1$,
$${\rm sdepth} (I^k) \leq {\rm sdepth} (I)$$ and $${\rm sdepth} (S/I^k) \leq {\rm sdepth} (S/I).$$
\end{cor}

The following example from \cite{h} shows that the inequalities of Corollary \ref{nor} do not necessarily hold, if $I$ is not a normal ideal.

\begin{exmp} \label{ex3}
Let $I=(x_1^4, x_1^3x_2, x_1x_2^3, x_2^4, x_1^2x_2^2x_3)$ be a monomial ideal in the polynomial ring $S=\mathbb{K}[x_1, x_2, x_3]$. Then ${\rm depth} (S/I) = 0$ and ${\rm depth} (S/I^2) = 1$. It follows from \cite[Proposition 2.13]{br} that ${\rm sdepth} (S/I) = 0$ and ${\rm sdepth} (S/I^2)\geq 1$
\end{exmp}

Let $I\subset S$ be a normal monomial ideal and assume that Stanley's conjecture is true. Then Corollary \ref{nor} implies that ${\rm depth} (I^k) \leq {\rm sdepth} (I)$ and ${\rm depth} (S/I^k) \leq {\rm sdepth} (S/I)$, for every integer $k\geq 1$. In particular,
$$\lim_{k\rightarrow \infty}{\rm depth} (I^k) \leq {\rm sdepth} (I)$$ and $$\lim_{k\rightarrow \infty}{\rm depth} (S/I^k) \leq {\rm sdepth} (S/I).$$
Since $I$ is normal, by \cite[Theorem 10.3.2]{hh'} and \cite[Theorem 3.3.18]{v'}, we have  $$\lim_{k\rightarrow \infty}{\rm depth} (S/I^k)=n-\ell(I),$$ where $\ell(I)$ is the analytic spread of $I$. Therefore, Stanley's conjecture implies that for every normal monomial ideal $I$, the inequalities $${\rm sdepth}(S/I)\geq n-\ell(I)$$ and $${\rm sdepth} (I)\geq n-\ell(I)+1$$
hold. These inequalities have been proved for some special classes of monomial ideals. In \cite{psy2}, the authors prove that if $I\subset S$ is a weakly polymatroidal ideal (see \cite[Definition 12.7.1]{hh'}), which is generated in the same degree, then ${\rm sdepth}(S/I)\geq n-\ell(I)$ and ${\rm sdepth}(I)\geq n-\ell(I)+1$. In \cite{psy1} the authors study the Stanley depth of powers of edge ideal of forest graphs. Let $G=(V,E)$ be a forest graph with $n$ vertices and $p$ connected components and let
$$I(G)=(x_ix_j: v_iv_j \in E )$$
be the edge ideal of $G$. Then ${\rm sdepth}(S/I(G)^k)\geq p$, for every integer $k\geq 1$ (\cite[Theorem 2.7]{psy1}). But it is known and easy to prove that for every forest with $n$ vertices and $p$ connected components, $\ell(I(G))=n-p$ (see \cite{v}, page 50 for more details), which means that ${\rm sdepth}(S/I(G)^k)\geq n-\ell(I)$ for every integer $k\geq 0$.

The following example shows that these inequalities do not hold for an arbitrary monomial ideal.

\begin{exmp} \label{example}
Consider the ideal $I=(x_1^2, x_2^2, x_1x_2x_3, x_1x_2x_4)\subset S=\mathbb{K}[x_1, x_2, x_3, x_4]$ and let $J=(x_1^2, x_2^2)\subset S$. Now $J\subset I$ and $JI=I^2$. Therefore, $J$ is a reduction of $I$. Since ${\rm ht}(I)=2$, according to \cite[Corollaries 8.2.5 and 8.3.9]{hs}, we conclude that $\ell(I)=2$. But $\mathfrak{m}=(x_1, x_2,x_3, x_4)$ is the associated prime of $S/I$ and hence by \cite[Proposition 1.3]{hvz} (see also \cite{a}), ${\rm sdepth}(S/I)=0$ and by \cite[Corollary 1.2]{i}, ${\rm sdepth}(I)\leq 2$. This shows that the inequalities ${\rm sdepth}(S/I)\geq n-\ell(I)$ and ${\rm sdepth} (I)\geq n-\ell(I)+1$ do not hold for $I$.
\end{exmp}

The ideal $I$, in Example \ref{example}, is not integrally closed. In fact the author has no example of integrally closed monomial ideals, for which the inequalities ${\rm sdepth}(S/I)\geq n-\ell(I)$ and ${\rm sdepth} (I)\geq n-\ell(I)+1$ do not hold. Therefore he presents the following conjecture.

\begin{conj} \label{conje}
Let $I\subset S$ be an integrally closed monomial ideal. Then ${\rm sdepth}(S/I)\geq n-\ell(I)$ and ${\rm sdepth} (I)\geq n-\ell(I)+1$.
\end{conj}

Assuming the conjecture is true, it follows together with Burch's inequality that Stanley's conjecture holds for $I^k$ and $S/I^k$ for $k\gg 0$, provided that $I$ is a normal ideal.

It is clear that for every monomial ideal $I$ in $S$, the Stanley depth of $I$ is at most $n$. Hence the Stanley depth of infinitely many powers of
$I$ is constant. The following corollary gives a refinement of this fact in the case of normal ideals.

\begin{cor}
Let $I\subset S$ be a normal monomial ideal. Then the following statements hold.
\begin{itemize}
\item[(i)] There exists an integer $k\geq 1$ such that for every integer $s\geq 1$,
 we have ${\rm sdepth} (I^k) = {\rm sdepth} (I^{sk})$.

\item[(ii)] There exists an integer $k\geq 1$ such that for every integer $s\geq 1$,
 we have ${\rm sdepth} (S/I^k) = {\rm sdepth} (S/I^{sk})$.
\end{itemize}
\end{cor}

\begin{proof}
(i) Let $k\geq 1$ be an integer such that ${\rm sdepth} (I^k) =\min_t \{{\rm sdepth} (I^t)\}$. Since $I^k$ is a normal ideal, Corollary \ref{nor}, implies that for every integer $s\geq 1$,
 we have $${\rm sdepth} (I^k) \geq {\rm sdepth} (I^{sk})$$and thus by the choice of $k$,
 $${\rm sdepth} (I^k) = {\rm sdepth} (I^{sk}).$$

(ii) The proof is similar to the proof of (i).
\end{proof}

Let $I$ be a monomial ideal. In the following theorem we compare the Stanley depth of $\overline{I}$ and the Stanley depth of powers of $I$. We will use this result in Section \ref{sec3}, to give a new proof for a result of Ratliff in the case of monomial ideals (see Theorem \ref{ass}).

\begin{thm} \label{main}
Let $I_2\subseteq I_1$ be two monomial ideals in $S$. Then there  exists an integer $k\geq 1$, such that for every $s\geq 1$
$${\rm sdepth} (I_1^{sk}/I_2^{sk}) \leq {\rm sdepth} (\overline{I_1}/\overline{I_2}).$$
\end{thm}
\begin{proof}
Note that by Remark \ref{clo}, there exist integers  $k_1, k_2\geq 1$, such that for every monomial $u\in S$, we have $u^{k_1}\in I_1^{k_1}$ (resp. $u^{k_2}\in I_2^{k_2}$) if and only if $u\in \overline{I_1}$ (resp. $u\in \overline{I_2}$). Let $k={\rm lcm}(k_1,k_2)$ be the least common multiple of $k_1$ and $k_2$. Then for every monomial $u\in S$, we have $u^k\in I_1^k$ (resp. $u^k\in I_2^k$) if and only if $u\in \overline{I_1}$ (resp. $u\in \overline{I_2}$). Hence for every monomial $u\in S$ and every $s\geq 1$, we have $u^{sk}\in I_1^{sk}$ (resp. $u^{sk}\in I_2^{sk}$) if and only if $u\in \overline{I_1}$ (resp. $u\in \overline{I_2}$). Now we prove that for this choice of $k$ and for every $s\geq 1$
$${\rm sdepth} (I_1^{sk}/I_2^{sk}) \leq {\rm sdepth} (\overline{I_1}/\overline{I_2}),$$
and this proves our assertion.

Let $$\mathcal{D} : I_1^{sk}/I_2^{sk}=\bigoplus_{i=1}^m t_i \mathbb{K}[Z_i]$$
be a Stanley decomposition of $I_1^{sk}/I_2^{sk}$, such that ${\rm sdepth}(\mathcal{D})={\rm sdepth} (I_1^{sk}/I_2^{sk})$. By the argument above, for every monomial $u\in \overline{I_1}\setminus \overline{I_2}$, we have $$u^{sk}\in I_1^{sk}\setminus I_2^{sk}.$$
Now for each monomial $u\in \overline{I_1}\setminus \overline{I_2}$ we define $Z_u:=Z_i$ and $t_u:=t_i$, where $i\in\{1, \ldots, m\}$ is the uniquely determined index such that $u^{sk}\in t_i \mathbb{K}[Z_i]$. It is clear that $\overline{I_1}/\overline{I_2}\subseteq \sum u\mathbb{K}[Z_u]$, where the sum is taken over all monomials $u\in \overline{I_1}\setminus \overline{I_2}$. For the converse inclusion note that for every $u\in \overline{I_1}\setminus \overline{I_2}$ and every $h\in \mathbb{K}[Z_u]$, clearly we have $uh\in \overline{I}$. By the choice of $t_u$ and $Z_u$, we conclude $u^{sk}\in t_u \mathbb{K}[Z_u]$ and therefore $u^{sk}h^{sk}\in t_u \mathbb{K}[Z_u]$. This implies that $u^{sk}h^{sk}\notin I_2^{sk}$ and as  argument above shows and by the choice of $k$, we have $uh\notin \overline{I_2}$. Therefore $$\overline{I_1}/\overline{I_2}= \sum u\mathbb{K}[Z_u],$$ where the sum is taken over all monomials $u\in \overline{I_1}\setminus \overline{I_2}$.

Now for every $1\leq i \leq m$, let
$$U_i=\{u\in \overline{I_1}\setminus\overline{I_2}: Z_u=Z_i \ {\rm and} \ t_u=t_i\}.$$
Without lose of generality we may assume that $U_i\neq \emptyset$ for every $1\leq i \leq l$ and $U_i=\emptyset$ for every $l+1\leq i\leq m$. For every $1\leq i \leq l$, let $u_i$ be the greatest common divisor of elements of $U_i$. Note that $$\overline{I_1}/\overline{I_2}= \sum_{i=1}^l \sum u\mathbb{K}[Z_i],$$
where the second sum is taken over all monomials $u\in U_i$. Since for every $u\in U_i$, $u^{sk}\in t_i \mathbb{K}[Z_i]$, it follows that $u_i^{sk}\in t_i \mathbb{K}[Z_i]$. Therefore $u_i \in \overline{I_1}\setminus \overline{I_2}$ and hence $u_i\in U_i$. Now for every $u\in U_i$, we have $u\mathbb{K}[Z_i]\subseteq u_i\mathbb{K}[Z_i]$ and thus $$\sum_{u\in U_i} u\mathbb{K}[Z_i]=u_i\mathbb{K}[Z_i].$$
It follows that $$\overline{I_1}/\overline{I_2}= \sum_{i=1}^l u_i\mathbb{K}[Z_i].$$

Next we prove that for every $1\leq i,j \leq l$ with $i\neq j$, the summands $u_i\mathbb{K}[Z_i]$ and $u_j\mathbb{K}[Z_j]$ intersect trivially. By contradiction let $v$ be a monomial in $u_i\mathbb{K}[Z_i]\cap u_j\mathbb{K}[Z_j]$. Then there exist $h_i\in \mathbb{K}[Z_i]$ and $h_j\in \mathbb{K}[Z_j]$, such that $u_ih_i=v=u_jh_j$. Therefore $u_i^{sk}h_i^{sk}=v^{sk}=u_j^{sk}h_j^{sk}$. But $u_i\in U_i$ and hence $u_i^{sk}\in t_i \mathbb{K}[Z_i]$, which implies that $u_i^{sk}h_i^{sk}\in t_i \mathbb{K}[Z_i]$. Similarly $u_j^{sk}h_j^{sk}\in t_j \mathbb{K}[Z_j]$. Thus $$v^{sk}\in t_i \mathbb{K}[Z_i]\cap t_j \mathbb{K}[Z_j],$$ which is a contradiction, because $\bigoplus_{i=1}^m t_i \mathbb{K}[Z_i]$ is a Stanley decomposition of $I_1^{sk}/I_2^{sk}$. Therefore  $$\overline{I_1}/\overline{I_2}=\bigoplus_{i=1}^l u_i \mathbb{K}[Z_i]$$
is a Stanley decomposition of $\overline{I_1}/\overline{I_2}$ which proves ${\rm sdepth} (\overline{I_1}/\overline{I_2})\geq \min_{i=1}^l|Z_i|\geq{\rm sdepth} (I_1^{sk}/I_2^{sk})$.
\end{proof}

We illustrate the procedure  of the proof of the Theorem \ref{main} in the following example.

\begin{exmp}
Let $I=(x_1^2x_2^2, x_1^2x_3^2, x_2^2x_3^2)$ be a monomial ideal in the polynomial ring $S=\mathbb{K}[x_1, x_2, x_3]$. Using \cite[Theorem 2.2]{j}, it follows that
$$\overline{I}=(x_1^2x_2^2, x_1^2x_3^2, x_2^2x_3^2, x_1^2x_2x_3, x_1x_2^2x_3, x_1x_2x_3^2).$$

It is clear that for every monomial $u\in \overline{I}$, we have $u^2\in I^2$. One can easily see that
$$\mathcal{D} : I^2=x_1^4x_2^4\mathbb{K}[x_1, x_2]\oplus x_1^4x_3^4\mathbb{K}[x_1, x_3]\oplus x_2^4x_3^4\mathbb{K}[x_2, x_3]$$
$$\oplus x_1^4x_2^4x_3\mathbb{K}[x_1, x_2]\oplus x_1^4x_2x_3^4\mathbb{K}[x_1, x_3]\oplus x_1x_2^4x_3^4\mathbb{K}[x_2, x_3]$$
$$\oplus x_1^4x_2^2x_3^2\mathbb{K}[x_1, x_2]\oplus x_1^2x_2^4x_3^2\mathbb{K}[x_2, x_3]\oplus x_1^2x_2^2x_3^4\mathbb{K}[x_1, x_3]$$
$$\oplus x_1^4x_2^2x_3^3\mathbb{K}[x_1, x_2]\oplus x_1^3x_2^4x_3^2\mathbb{K}[x_2, x_3]\oplus x_1^2x_2^3x_3^4\mathbb{K}[x_1, x_3]$$
$$\oplus x_1^4x_2^4x_3^4\mathbb{K}[x_1, x_2, x_3]$$
is a Stanley decomposition of $I^2$ and indeed ${\rm sdepth} (I^2)=2$. Now we construct a Stanley decomposition $\mathcal{D}'$ for $\overline{I}$, with ${\rm sdepth} (\mathcal{D}')=2$.
Note that there is no monomial $u\in \overline{I}$, such that
$$u^2\in x_1^4x_2^4x_3\mathbb{K}[x_1, x_2]\oplus x_1^4x_2x_3^4\mathbb{K}[x_1, x_3]\oplus x_1x_2^4x_3^4\mathbb{K}[x_2, x_3]$$
$$\oplus x_1^4x_2^2x_3^3\mathbb{K}[x_1, x_2]\oplus x_1^3x_2^4x_3^2\mathbb{K}[x_2, x_3]\oplus x_1^2x_2^3x_3^4\mathbb{K}[x_1, x_3].$$
Now the greatest common divisor of monomials $u \in \overline{I}$ with $u^2\in x_1^4x_2^4\mathbb{K}[x_1, x_2]$ is equal to $x_1^2x_2^2$. Therefore the Stanley space $x_1^2x_2^2\mathbb{K}[x_1, x_2]$ appears as a direct summand in our desired Stanley decomposition. Similarly $x_1^2x_3^2\mathbb{K}[x_1, x_3], x_2^2x_3^2\mathbb{K}[x_2, x_3], x_1^2x_2x_3\mathbb{K}[x_1, x_2],$ $x_1x_2^2x_3\mathbb{K}[x_2, x_3], x_1x_2x_3^2\mathbb{K}[x_1, x_3]$ and $x_1^2x_2^2x_3^2\mathbb{K}[x_1, x_2, x_3]$ are the other Stanley spaces in our desired Stanley decomposition. Therefore we derive the following Stanley decomposition for $\overline{I}$.
$$\mathcal{D}' : \overline{I}=x_1^2x_2^2\mathbb{K}[x_1, x_2]\oplus x_1^2x_3^2\mathbb{K}[x_1, x_3]\oplus x_2^2x_3^2\mathbb{K}[x_2, x_3]$$
$$\oplus x_1^2x_2x_3\mathbb{K}[x_1, x_2]\oplus x_1x_2^2x_3\mathbb{K}[x_2, x_3]\oplus x_1x_2x_3^2\mathbb{K}[x_1, x_3]$$
$$\oplus x_1^2x_2^2x_3^2\mathbb{K}[x_1, x_2, x_3]$$
This shows that ${\rm sdepth} (\overline{I})\geq2$ and in fact one can easily see that ${\rm sdepth} (\overline{I})=2$.

\end{exmp}

\begin{cor} \label{mainc}
Let $I\subset S$ be a monomial ideal. Then there exist integers $k_1, k_2 \geq 1$, such that for every $s\geq 1,$
$${\rm sdepth} (I^{sk_1}) \leq {\rm sdepth} (\overline{I})$$
and
$${\rm sdepth} (S/I^{sk_2}) \leq {\rm sdepth} (S/\overline{I}).$$
\end{cor}

Corollary \ref{mainc} shows that if $I$ is an integrally closed monomial ideal, then there exists an integer $k$ such that for every $s\geq 1$ the inequality ${\rm sdepth} (S/I^{sk}) \leq {\rm sdepth} (S/I)$ holds. The following example shows that this inequality does not necessarily hold if $I$ is not integrally closed.

\begin{exmp}
Let $I=(x_1^4, x_1^3x_2, x_1x_2^3, x_2^4, x_1^2x_2^2x_3)$ be a monomial ideal in the polynomial ring $S=\mathbb{K}[x_1, x_2, x_3]$. Then
$$I^2=(x_1^8, x_1^7x_2, x_1^5x_2^3, x_1^4x_2^4, x_1^6x_2^2, x_1^3x_2^5, x_1^2x_2^6, x_1x_2^7, x_2^8).$$
Now the variable $x_3$ does not divide the minimal generators of $I^2$. Hence it does not divide the minimal generators of $I^{2t}=(I^2)^t$, for every integer $t\geq 1$. Therefore the maximal ideal $\mathfrak{m}=(x_1,x_2,x_3)$ of $S$ is not an associated prime of $S/I^{2t}$ and so \cite[Proposition 2.13]{br} implies that ${\rm sdepth}(S/I^{2t})\geq 1$, for every integer $t\geq 1$. But as we mentioned in Example \ref{ex3}, ${\rm sdepth}(S/I)= 0$. This shows that there does not exist any integer $k$ such that for every $s\geq 1$ the inequality ${\rm sdepth} (S/I^{sk}) \leq {\rm sdepth} (S/I)$ holds
\end{exmp}

The following corollaries are immediate consequences of Theorem \ref{main} and Corollary \ref{mainc}.

\begin{cor}
Let $I_2\subseteq I_1$ be two monomial ideals in $S$. Then
$$\min_k \{{\rm sdepth} (I_1^k/I_2^k)\} \leq {\rm sdepth} (\overline{I_1}/\overline{I_2}).$$
\end{cor}

\begin{cor}
Let $I\subset S$ be a monomial ideal. Then
$$\min_k \{{\rm sdepth} (S/I^k)\} \leq {\rm sdepth} (S/\overline{I}),$$
and
$$\min_k \{{\rm sdepth} (I^k)\} \leq {\rm sdepth} (\overline{I}).$$
\end{cor}


\section{An application of Stanley depth} \label{sec3}

Let $I$ be a monomial ideal of the polynomial ring $S=\mathbb{K}[x_1,\dots,x_n]$. In this section we will examine the sets
of associated primes of the powers of $I$, that is, the sets
$${\rm Ass}(S/I^k) = \{P\subset S : P {\rm \ is \ prime \ and\ } P = (I^k : c) {\rm \ for\ some\ } c\in S\}, \ \ \ \  k \geq 1.$$
Since $I$ is a monomial ideal of a polynomial ring $S$, the associated primes will be monomial
primes, which are primes that are generated by subsets of the variables, see \cite[Corollary 1.3.9]{hh'}.
In \cite{b'}, Brodmann showed that the sets ${\rm Ass}(S/I^k)$ stabilize for large $k$. That is, there exists
a positive integer $N_1\geq 1$ such that ${\rm Ass}(S/I^k) = {\rm Ass}(S/I^{N_1})$ for all $k\geq N_1$.
We denote the set ${\rm Ass}(S/I^{N_1})$ by ${\rm Ass}_{\infty}(S/I)$. Ratliff studied the set of
associated primes of the integral closure of powers of ideals. By his results \cite{r1,r2}, one has that
 the sets ${\rm Ass}(S/\overline{I^k})$ form an ascending chain which
stabilizes for large $k$. Thus, there exists $N_2\geq 1$ such that ${\rm Ass}(S/\overline{I^k})
 = {\rm Ass}(S/\overline{I^{N_2}})$ for all $k\geq N_2$. We denote the set ${\rm Ass}(S/\overline{I^{N_2}})$ by ${\rm \overline{Ass}}_{\infty}(S/I)$. The
set ${\rm \overline{Ass}}_{\infty}(S/I)$ is nicely described in \cite{m}. It is known \cite[Theorem 2.8]{r2} that the inclusion ${\rm \overline{Ass}}_{\infty}(S/I)\subseteq {\rm Ass}_{\infty}(S/I)$ holds for any ideal $I$ of a commutative Noetherian ring (see \cite[Proposition 3.17]{m'} for additional details). As an application of Corollary \ref{mainc} we give a new proof for this result in the case of monomial ideals. To the best of my knowledge, this would be the first application
of Stanley depth.

First we need to introduce some notation and basic facts.

Let $P = (x_{i_1}, \ldots, x_{i_r})$ be a monomial prime ideal in $S$, and $I\subseteq S$ any monomial ideal and let $L=[n]\setminus\{x_{i_1}, \ldots, x_{i_r}\}$. We denote by $I(P)$ the monomial ideal in the polynomial ring $S(P) = \mathbb{K}[x_{i_1}, \ldots, x_{i_r}]$, which is obtained from $I$ by applying the $\mathbb{K}$-algebra homomorphism $S\rightarrow S(P)$ with $x_i\mapsto 1$ for all $i\in L$. It is known that (\cite[Lemma 1.3]{hrv})
$${\rm Ass}(S(P)/I(P))=\{Q\in {\rm Ass}(S/I): x_i\notin Q {\rm \ for \ all \ } i\in L\}.$$
We use this simple fact for proving Theorem \ref{ass}.

We also need the following simple lemma.

\begin{lem} \label{simple}
For every monomial ideal $I$ and every monomial prime ideal $P=(x_{i_1}, \ldots, x_{i_r})$ of the polynomial ring $S=\mathbb{K}[x_1,\dots,x_n]$, we have $\overline{I(P)}=\overline{I}(P)$, as ideals of the polynomial ring $S(P) = \mathbb{K}[x_{i_1}, \ldots, x_{i_r}]$.
\end{lem}

\begin{proof}
It is clear that $\overline{I}(P)\subseteq \overline{I(P)}$. Hence it suffices to prove the converse inclusion. Let $L=[n]\setminus\{x_{i_1}, \ldots, x_{i_r}\}$. Without loss of generality we may assume that $L=\{x_1, \ldots, x_{n-r}\}$. For every monomial $u\in \overline{I(P)}$, there exists an integer $k\geq 1$, such that $u^k\in I(P)^k$. Then for sufficiently large integers $l_1, \ldots l_{n-r}$, we have $(ux_1^{l_1}\ldots x_{n-r}^{l_{n-r}})^k= u^kx_1^{kl_1}\ldots x_{n-r}^{kl_{n-r}}\in I^k$. Hence $ux_1^{l_1}\ldots x_{n-r}^{l_{n-r}}\in \overline{I}$ and thus $u\in \overline{I}(P)$, which implies that $\overline{I(P)}=\overline{I}(P)$.
\end{proof}

Now we are ready to give a new proof for the theorem of Ratliff (\cite[Theorem 2.8]{r2}) in the case of monomial ideals.

\begin{thm} \label{ass}
\cite[Theorem 2.8]{r2} Let $I \subseteq S$ be a monomial ideal. Then ${\rm \overline{Ass}}_{\infty}(S/I)\subseteq {\rm Ass}_{\infty}(S/I)$.
\end{thm}
\begin{proof}
Let $N_1, N_2\geq 1$ be two integers such that $${\rm Ass}(S/I^{N_1})={\rm Ass}(S/I^{N_1+k})$$ and $${\rm Ass}(S/\overline{I^{N_2}})={\rm Ass}(S/\overline{I^{N_2+k}}),$$ for every integer $k\geq 1$. Let $P\in {\rm Ass}(S/\overline{I^{N_2}})$ be a monomial prime ideal of $S$. Then by \cite[Lemma 1.3]{hrv}, we have $$P\in {\rm Ass}(S(P)/\overline{I^{N_2}}(P))= {\rm Ass}(S(P)/\overline{I^{N_2}(P)})= {\rm Ass}(S(P)/\overline{I(P)^{N_2}}),$$ where the first equality follows from Lemma \ref{simple} and the second equality is trivial. Since $P$ is the maximal ideal of $S(P)$, It follows from \cite[Proposition 1.3]{hvz} (see also \cite{a}) that $${\rm sdepth}_{S(P)}(S(P)/\overline{I(P)^{N_2}})=0.$$
By Corollary \ref{mainc}, there exists an integer $l\geq N_1$ such that
$${\rm sdepth}_{S(P)}(S(P)/I(P)^l)\leq {\rm sdepth}_{S(P)}(S(P)/\overline{I(P)^{N_2}}),$$
and therefore $${\rm sdepth}_{S(P)}(S(P)/I(P)^l)=0.$$ Thus according to \cite[Proposition 2.13]{br}, $P$ is an associated prime of $S(P)/I(P)^l$. Now by \cite[Lemma 1.3]{hrv}, $$P\in {\rm Ass}(S/I^l)={\rm Ass}(S/I^{N_1}).$$ Thus
$${\rm \overline{Ass}}_{\infty}(S/I)={\rm Ass}(S/\overline{I^{N_2}})\subseteq {\rm Ass}(S/I^{N_1})={\rm Ass}_{\infty}(S/I),$$
and this completes the proof of the theorem.
\end{proof}



\section*{Acknowledgments}
This work was done while the author visited Philipps-Universit${\rm \ddot{a}}$t Marburg supported
by DAAD. The author thanks  J${\rm \ddot{u}}$rgen Herzog and Volkmar Welker for useful discussions
during the preparation of the article. He is also grateful to Irena Swanson for generously sharing her knowledge about integral closure.
 He also thanks Siamak Yassemi for reading an earlier version of this article and for his helpful comments. The author would like to thank the referee for his/her careful reading of the paper
and for his/her valuable comments.



\end{document}